\documentclass[a4paper,11pt,DIV=11,oneside,headings=small,bibliography=totoc,abstract=true,intlimits, draft=true]{scrartcl}
\usepackage{amssymb,amsmath,amsfonts,amsthm}
\allowdisplaybreaks
\usepackage{mathtools}
\usepackage[english]{babel}
\usepackage[T1]{fontenc}
\usepackage{textgreek}
\usepackage{enumerate}
\usepackage{dsfont}
\usepackage[noblocks]{authblk}

\usepackage[textsize=tiny,shadow]{todonotes}

\usepackage[linktocpage, pdftex,colorlinks]{hyperref}
	\usepackage{color}
 	\definecolor{darkred}{rgb}{0.5,0,0}
 	\definecolor{darkgreen}{rgb}{0,0.5,0}
 	\definecolor{darkblue}{rgb}{0,0,0.5}	
	\hypersetup{colorlinks,linkcolor=black,filecolor=black,urlcolor=black,citecolor=black}
\usepackage[capitalise, noabbrev]{cleveref} 
\usepackage{colortbl}
\usepackage{booktabs}

\newcommand{\thickset}{E}

\newcommand{\drm}{\mathrm{d}}
\newcommand{\rmd}{\mathrm{d}}
\newcommand{\euler}{\mathrm{e}}

\newcommand{\N}{\mathds{N}}
\newcommand{\R}{\mathds{R}}
\newcommand{\C}{\mathds{C}}

\newcommand{\cL}{\mathcal{L}}
\newcommand{\F}{\mathcal{F}}

\newcommand{\calL}{\mathcal{L}}
\newcommand{\fourier}{\mathcal{F}}
\newcommand{\1}{\mathds{1}}
\newcommand{\from}{\colon}
\newcommand{\norm}[1]{\left\lVert#1\right\rVert}
\newcommand{\abs}[1]{\left\lvert#1\right\rvert}
\newcommand{\re}{\operatorname{Re}}
\renewcommand{\Re}{\re}

\newcommand{\supp}{\operatorname{supp}}
\newcommand{\id}{\operatorname{Id}}

\DeclareMathOperator*{\esssup}{ess\,sup}
\newcommand{\e}{\textrm{e}}
\newcommand{\ip}[2]{\left(#1|#2\right)}
\providecommand{\differential}{\mathrm{d}}
\renewcommand{\epsilon}{\varepsilon}
\renewcommand{\d}{\differential}
\newtheorem{theorem}{Theorem}[section]
\newtheorem{proposition}[theorem]{Proposition}
\newtheorem{lemma}[theorem]{Lemma}
\newtheorem{corollary}[theorem]{Corollary}
\theoremstyle{definition}
\newtheorem{example}[theorem]{Example}
\newtheorem{definition}[theorem]{Definition}

\theoremstyle{remark}
\newtheorem{remark}[theorem]{Remark}

\DeclareMathOperator{\sgn}{sgn}

\title{Final State Observability in Banach spaces with applications to Subordination and Semigroups induced by L{\'e}vy processes}

\author[1]{Dennis Gallaun}

\author[2,3]{Jan Meichsner}

\author[1]{Christian Seifert}

\affil[1]{Technische Universit\"at Hamburg, Institut f\"ur Mathematik, Am Schwarzenberg-Campus 3, 21073 Hamburg, Germany, \{dennis.gallaun, christian.seifert\}@tuhh.de}
\affil[2]{FernUniversit\"at in Hagen, Lehrgebiet Analysis, Fakultät Mathematik und Informatik, Universit\"atsstra{\ss}e 47,58084 Hagen, Germany}
\affil[3]{Technische Universit\"at Dresden, Arbeitsgruppe Astronomie, Lohrmann-Observatorium, August-Bebel-Stra{\ss}e 30, 01219 Dresden, Germany, jan.meichsner@tu-dresden.de}

\date{\vspace{-7ex}}

\hypersetup{final}

\begin{document}

\maketitle

\begin{abstract}
This paper generalizes the abstract method of proving an observability estimate by combining an uncertainty principle and a dissipation estimate. In these estimates we allow for a large class of growth/decay rates satisfying an integrability condition. In contrast to previous results, we use an iterative argument which enables us to give an asymptotically sharp estimate for the observation constant and which is explicit in the model parameters. 
We give two types of applications where the extension of the growth/decay rates naturally appear. By exploiting subordination techniques we show how the dissipation estimate of a semigroup transfers to subordinated semigroups. Furthermore, we apply our results to semigroups related to L{\'e}vy processes.
\\[1ex]
\\[1ex]
\end{abstract}
\section{Introduction}

Let $X,Y$ be Banach spaces, $A$ a densely defined closed linear operator in $X$ such that $A$ generates a $C_0$-semigroup $\bigl(S(t)\bigr)_{t\geq0}$, $C\in\calL(X,Y)$, and $T>0$. We consider the observation system
\begin{align*}
  \dot{x}(t) & = Ax(t) \quad(t\in(0,T]),\qquad x(0)=x_0\in X,\\
  y(t) & = Cx(t) \quad(t\in[0,T]).
\end{align*}
Note that for $x_0\in X$ the state function $x\from[0,T]\to X$ is given by the mild solution $x(t) = S(t)x_0$ for $t\in[0,T]$, and the observation function $y\from [0,T]\to Y$ is given by $y(t) = Cx(t) = CS(t)x_0$ for $t\in[0,T]$.
For given $r\in[1,\infty]$ we are interested in a \emph{final state observability estimate} for the system, i.e., the existence of an observability constant $C_{\mathrm{obs}}\geq 0$ such that for all $x_0\in X$ we have
\begin{equation}
  \label{eq:OBS}
  \norm{S(T)x_0}_X \leq C_{\mathrm{obs}} \begin{cases} 
  \Bigl(\int_{0}^T \norm{CS(t)x_0}_Y^r \,\rmd t\Bigr)^{1/r} & r\in[1,\infty),\\
  \esssup_{t\in[0,T]} \norm{CS(t)x_0}_Y & r=\infty.
  \tag{OBS}
  \end{cases}
\end{equation}
Put differently, the final state $x(T)$ can be estimated by knowing only the observations $y(t)$ for $t\in[0,T]$. Observe that the formulation of \eqref{eq:OBS} just requires a semigroup with suitable measurability properties, namely measurability of $t\mapsto \norm{CS(t)x_0}_Y$.

A by now well-known method (see \cite{Miller-10,TenenbaumT-11,BeauchardP-18,NakicTTV-20,GallaunST-20}) to obtain a final state observability estimate is given by establishing a so-called \emph{uncertainty principle}, sometimes also called \emph{spectral inequality},
\begin{equation}
  \label{eq:UP}
  \forall \lambda>0,\, x\in X: \quad \norm{P_\lambda x}_X \leq d_0\euler^{d_1 \lambda^{\gamma_1}} \norm{C P_\lambda x}_Y
  \tag{UP}
\end{equation}
and a related \emph{dissipation estimate}
\begin{equation}
  \label{eq:DISS}
  \forall \lambda>0, \, t\in(0,T],\, x\in X: \quad \norm{(I-P_\lambda)S(t)x}_X \leq d_2\euler^{-d_3 \lambda^{\gamma_2} t^{\gamma_3}} \norm{x}_X,
  \tag{DISS}
\end{equation}
where $(P_\lambda)_{\lambda>0}$ is a family of bounded linear operators on $X$ and $d_0,d_1,d_2,d_3,\gamma_1,\gamma_2,\gamma_3>0$ with $\gamma_1<\gamma_2$.
Then $C_{\mathrm{obs}}$ can be explicitly determined of the form
\[C_{\mathrm{obs}} = \frac{C_1}{T^{1/r}}\exp\left(\frac{C_2}{T^{\frac{\gamma_1\gamma_3}{\gamma_2-\gamma_1}}}+C_3 T\right)\]
with explicitly known constants $C_1,C_2,C_3\geq 0$, see \cite{NakicTTV-20,GallaunST-20}.

In mathematical systems theory, also control systems are studied. There, given another Banach space $U$ and a bounded linear operator $B\in\calL(U,X)$, one considers the control system
\begin{align*}
  \dot{x}(t) & = Ax(t) + Bu(t) \quad(t\in(0,T]),\qquad x(0)=x_0\in X,
\end{align*}
and asks whether, for given $r\in[1,\infty]$, for every $x_0\in X$ and $\epsilon >0$ there exists a control function $u\in L_r([0;T];U)$ such that we have $\|x(T)\| < \epsilon$. Then the control system is called \emph{approximate null-controllable}. If, furthermore, there exists $C_{\mathrm{unc}}\geq 0$ such that we can choose the control function $u$ to satisfy $\norm{u}_{L_r((0;T);U)}\leq C_{\mathrm{unc}}\norm{x_0}_X$ for all $x_0\in X$ and $\epsilon >0$, then the approximate null-controllability is called \emph{cost-uniform}. Note that for reflexive Banach spaces this is equivalent to the standard concept of \emph{null-controllability}. 
By means of duality \cite{Douglas-66,DoleckiR-77,Carja-88,Vieru-05}, cost-uniform approximate null-controllability is equivalent to a final state observability estimate for the corresponding dual system
\begin{align*}
  \dot{x}(t) & = A'x(t) \quad(t\in(0,T]),\qquad x(0)=x_0\in X',\\
  y(t) & = B'x(t) \quad(t\in[0,T]),
\end{align*}
and $C_{\mathrm{unc}}$ agrees with the corresponding observability constant. 
Therefore, in order to prove null-controllability properties, we only need to establish a final state observability estimate of the form \eqref{eq:OBS} for the dual semigroup $(S(t)')_{t\geq 0}$ of the $C_0$-semigroup generated by $A$.

In this paper we focus on the above-mentioned well-established method to prove an observability estimate \eqref{eq:OBS} given an uncertainty principle \eqref{eq:UP} and a dissipation estimate \eqref{eq:DISS}, i.e.,
\[\eqref{eq:UP} + \eqref{eq:DISS} \implies \eqref{eq:OBS}.\]
We will work with general growth rates for $\eqref{eq:UP}$ and decay rates for \eqref{eq:DISS}, i.e., with increasing functions $f,g,h\from(0,\infty)\to(0,\infty)$ such that
\begin{equation}
  \label{eq:UPg}
  \forall x\in X, \, \lambda>0: \quad \norm{P_\lambda x}_X \leq \euler^{f(\lambda)} \norm{C P_\lambda x}_Y
  \tag{UPgen}
\end{equation}
and
\begin{equation}
  \label{eq:DISSg}
  \forall x\in X, \, \lambda>0, \, t\in(0,T]: \quad \norm{(I-P_\lambda)S(t)x}_X \leq \euler^{-g(\lambda)h(t)} \norm{x}_X.
  \tag{DISSgen}
\end{equation}
We show that under suitable assumptions on the functions $f,g$ and $h$ the method can still be applied to obtain a final state observability estimate, i.e.,
\[\eqref{eq:UPg} + \eqref{eq:DISSg} \implies \eqref{eq:OBS}.\]
Such a form of generalization has already been considered for the Hilbert space setting in \cite{DuyckaertsM-12}. However, to the best of our knowledge, the abstract point of view using semigroups on Banach spaces was not addressed so far. Moreover, compared to \cite{DuyckaertsM-12}, we provide an (asymptotically sharp) estimate for $C_{\mathrm{obs}}$ which is explicit in the model parameters. Furthermore, we also allow for more general functions $f,g$ and $h$ as in \cite{DuyckaertsM-12}. In this sense, we extend the so far known results. 

The paper is organised as follows. In Section \ref{sec:growth_rates} we state and prove the abstract theorem stating the implication
\[\eqref{eq:UPg} + \eqref{eq:DISSg} \implies \eqref{eq:OBS}.\]
Moreover, we relate the assumptions on $f,g$ and $h$ to the ones in the existing literature and give illustrative examples.
Then in Section \ref{sec:subordination} we exploit subordination techniques to show that \eqref{eq:UPg} and \eqref{eq:DISSg} do not just imply \eqref{eq:OBS} for one observation system, but also for many related ones. It turns out that only the function $g$ is influenced by subordination, namely if $\varphi$ is the corresponding Bernstein function and $\varphi\circ g$ satisfies the assumptions on $g$ instead, then the subordinated semigroup also satisfies a final state observability estimate. This fits nicely to the intuition since the subordinated semigroup is generated by $-\varphi(-A)$, where $A$ is the generator of the original semigroup. In Section \ref{sec:application} we will apply our results to semigroups induced by vaguely continuous convolution semigroups of sub-probability measures on $\R^n$, i.e., of semigroups related to L{\'e}vy processes. The function $g$ is then related to the corresponding L\'{e}vy measure, or put differently, the symbol of the semigroup of measures.
These two types of applications naturally give rise to the extension of the growth/decay rates to general functions.

\section{Sufficient Criteria for Final State Observability}
\label{sec:growth_rates}

The following theorem provides sufficient conditions for a final state observability estimate in terms of a generalized uncertainty principle and dissipation estimate. Moreover, the theorem gives an upper bound on the observability constant. 

\begin{theorem}\label{thm:spectral+diss-obs} 
Let $X$ and $Y$ be Banach spaces, $C\in \cL(X,Y)$, $(P_\lambda)_{\lambda > 0}$ in $\cL(X)$, $\bigl(S(t)\bigr)_{t\geq 0}$ a semigroup on $X$, $M \geq 1$ and $\omega \in \R$ such that $\lVert S(t) \rVert \leq M \euler^{\omega t}$ for all $t \geq 0$, and assume that for all $x\in X$ the mapping $t\mapsto \lVert C S(t) x \rVert_Y$ is measurable. Consider measurable functions $f,g,h\from (0,\infty) \to (0,\infty)$ such that $f$ and $h$ are strictly monotonically increasing and bijective, and $\lambda\mapsto \lambda / g(f^{-1}(\lambda))$ is monotonically decreasing. Let $T>0$ and assume that there exists $\lambda_T \geq 0$ with the property 
\begin{align} 
\int_{\lambda_T}^\infty h^{-1}\left(4\frac{\lambda}{g(f^{-1}(\lambda))}\right)\frac{1}{\lambda} \;\drm \lambda \le  \frac{\ln(2)}{4}\min\{1,T\}.
\label{eq:ass:g}
\end{align}
Furthermore, let $C_1\geq 0$, $C_2\geq 1$ and assume that
\begin{align} 
\forall \lambda >0,\, x\in X: \quad \lVert P_\lambda x \rVert_{ X } \le C_1 \euler^{f(\lambda)} \lVert C  P_\lambda x \rVert_{Y }
\label{eq:ass:uncertainty} , 
\end{align}
and
\begin{align} 
\forall \lambda>0,\, t\in (0,T],\, x\in X: \quad \lVert (\id-P_\lambda) S(t) x \rVert_{X} \le C_2 \euler^{-g(\lambda)h(t)+\omega t} \lVert x \rVert_{X} \label{eq:ass:dissipation}.
\end{align}
Then we have
\begin{equation} \label{eq:obs}
\forall x\in X: \quad \norm{S(T)x}_X\leq C_{\mathrm{obs}} \int_0^T \lVert C S(\tau) x \rVert_{Y} \drm \tau \quad \text{with}\quad C_{\mathrm{obs}} \le \frac{C_3}{T}\euler^{6\lambda_T  + \omega_+ T},
\end{equation}
where $\omega_+ = \max\{\omega,0\}$ and
$C_3 = 8 \euler^3 MC_1\left(M \left(C_1\lVert C \rVert_{\cL (X,Y)}+1\right)C_2\right)^{\frac{6}{\euler \ln(2)}}.$
\end{theorem}

The integrability condition \eqref{eq:ass:g} in Theorem \ref{thm:spectral+diss-obs}  describes the class of functions $f,g,h$ for which our strategy works. Moreover, the quantity $\lambda_T$ describes the $T$-dependence of $C_{\mathrm{obs}}$. In Example \ref{ex:function_g} appropriate values of $\lambda_T$ for concrete scenarios are computed. Note that in Theorem \ref{thm:spectral+diss-obs} we do not assume the semigroup $\bigl(S(t)\bigr)_{t\geq 0}$ to be strongly continuous, but only that the mapping $t\mapsto \lVert C S(t) x \rVert_Y$ is measurable. This makes the theorem also applicable for dual semigroups on non-reflexive Banach spaces.

The proof of Theorem \ref{thm:spectral+diss-obs} is given at the end of this section. Before, we mention direct consequences of Theorem \ref{thm:spectral+diss-obs} in Remark \ref{rem:modification_thm}. And in Remark \ref{rem:compare_DM-12} as well as Example \ref{ex:compare_GST} we compare our result to \cite{DuyckaertsM-12} and \cite{GallaunST-20}. 

\begin{remark}\label{rem:modification_thm}
\begin{enumerate}
\item \label{rem:mod_thm_c} We can use H\"older's inequality to obtain from \eqref{eq:obs} an observability estimate w.r.t.\ to the $L_r$-norm for $r\in [1,\infty]$. 
Indeed, applying Theorem \ref{thm:spectral+diss-obs} for the rescaled semigroup $\widetilde{S}(t) := \euler^{-\omega t} S(t)$ ($t\geq 0$) yields for $x\in X$ the estimate
\begin{align*}
\norm{S(T)x}_X = \euler^{\omega T} \|\widetilde{S}(T)x\|_X 
\leq \euler^{\omega T} C_{\mathrm{obs}} \int_0^T \euler^{-\omega \tau} \lVert C S(\tau) x \rVert_{Y} \, \d \tau. 
\end{align*}
Note that, since $\bigl(\widetilde S(t)\bigr)_{t\geq 0}$ is bounded, $C_{\mathrm{obs}}$ does not depend on $\omega$. 
Using H\"older's inequality, we obtain 
\begin{align*}
\|S(T)x\|_X \leq C \left(\int_0^T \lVert C S(\tau) x \rVert_{Y}^r \drm \tau\right)^{\frac{1}{r}} \quad \text{with } C = C_\mathrm{obs}
\begin{cases}
\left(\frac{\euler^{\omega r' T}-1}{\omega r'}\right)^{\frac{1}{r'}}  & \text{for } \omega >0 \\
T^{1/r'} & \text{for } \omega =0 \\
\left(\frac{1-\euler^{\omega r' T}}{-\omega r'}\right)^{\frac{1}{r'}} & \text{for } \omega <0,
\end{cases}
\end{align*}
and $r'\in [1,\infty]$ with $1/r + 1/r' = 1$.
\item \label{rem:mod_thm_b} Similar as in \cite{DuyckaertsM-12}, instead of \eqref{eq:ass:dissipation} we can use the dissipation estimate 
\begin{align*} 
\forall \lambda>0,\, t\in (0,T],\, x\in X: \quad \lVert (\id-P_\lambda) S_t x \rVert_{X} \le C_2 \euler^{-g(\lambda)h(t) + \omega_+ T +m\lambda} \lVert x \rVert_{X} 
\end{align*}
for some $m\geq 0$. In this case we get the final state observability estimate \eqref{eq:obs} with $\lambda_T \geq 0$ satisfying
\begin{align*} 
\int_{\lambda_T}^\infty h^{-1}\left((4+m)\frac{\lambda}{g(f^{-1}(\lambda))}\right)\frac{1}{\lambda} \;\drm \lambda \le  \frac{\ln(2)}{4}\min\{1,T\}.
\end{align*}
Although in Theorem \ref{thm:spectral+diss-obs} we consider $m=0$, the case $m>0$ can easily be implemented in our proof.
\end{enumerate}
\end{remark}


\begin{remark}\label{rem:compare_DM-12}
In \cite[Section 6]{DuyckaertsM-12} Duyckaerts and Miller proved a similar result to Theorem \ref{thm:spectral+diss-obs} by also allowing for more general growth/decay in the uncertainty principle and dissipation estimate. Let us relate Theorem \ref{thm:spectral+diss-obs} to the results obtained in \cite{DuyckaertsM-12}. First note that Duyckaerts and Miller allow for unbounded observation operators $C$ satisfying an admissibility condition, by using conditions involving a reference operator $C_0$, whereas as we stick to bounded $C$ and choose $C_0 = \id$. Furthermore, Duyckaerts and Miller restrict to Hilbert spaces and consider $h(t) = t$ for $t\in (0,T]$. In \cite{DuyckaertsM-12} an integrated version of \eqref{eq:ass:uncertainty} is considered:
\begin{align} 
\forall \lambda >\lambda_0,\,  t\in (0,T],\, x\in \mathcal{E}_\lambda: \quad \lVert S(t) x \rVert_{ X }^2 \le C_1 \frac{1}{t} \euler^{\lambda/\varphi(\lambda)} \int_0^t \lVert C  S(\tau) x \rVert_{Y }^2 \; \drm \tau
\label{eq:ass:uncertainty_DM} , 
\end{align}
where $\varphi$ is a positive, continuous, monotonically increasing function with $\lambda/\varphi(\lambda) \to \infty$ for $\lambda \to \infty$ and $(\mathcal{E}_\lambda)_\lambda$ is a non-decreasing family of semigroup invariant spaces. On the orthogonal complement of $\mathcal{E}_\lambda$, Duyckaerts and Miller consider the dissipation estimate 
\begin{align} 
\forall \lambda>\lambda_0,\, t\in (0,T],\,x\perp\mathcal{E}_\lambda: \quad \lVert  S(t) x \rVert_{X} \le C_2 \euler^{-\lambda t + m\lambda/\varphi(\lambda)} \lVert x \rVert_{X}, \label{eq:ass:dissipation_DM}
\end{align}
where $m\geq 0$. To turn \eqref{eq:ass:uncertainty_DM} and \eqref{eq:ass:dissipation_DM} into an observability estimate, the integrability of 
\begin{equation}
s\mapsto \frac{1}{\psi^{-1}(\varphi(q^s))}
\label{eq:integral_condition_DM}
\end{equation}
at $\infty$ is assumed, where $\psi(\lambda) = \lambda \ln(\lambda)$ ($\lambda>1$) and $q>1$.
The inequalities \eqref{eq:ass:uncertainty_DM} and \eqref{eq:ass:dissipation_DM} can be related to \eqref{eq:ass:uncertainty} and \eqref{eq:ass:dissipation} by means of the transformation $\varphi(\lambda) = \lambda/f(g^{-1}(\lambda))$ (assuming $g$ is invertible). Indeed, note that if $\lambda\mapsto \lambda/g(f^{-1}(\lambda))$ is monotonically decreasing then $\varphi$ is increasing. Moreover, the integrability of \eqref{eq:integral_condition_DM} at $\infty$ yields with $\psi^{-1}(t)\leq t$ for all $t\geq \e$ and $h^{-1}(t)=t$ for $t>0$ and for $R>0$ large enough (also assuming sufficiently smooth functions):
\begin{align*}
& \int_{\ln(\varphi^{-1}(R))/\ln(q)}^\infty \frac{1}{\psi^{-1}(\varphi(q^s))}\;\drm s 
\geq \int_{\ln(\varphi^{-1}(R))/\ln(q)}^\infty \frac{1}{\varphi(q^s)} \; \drm s 
= \frac{1}{\ln(q)} \int_{\varphi^{-1}(R)}^\infty \frac{1}{\varphi(\lambda) \lambda} \; \drm \lambda \\
& = \frac{1}{\ln(q)} \int_{\varphi^{-1}(R)}^\infty \frac{f(g^{-1}(\lambda))}{\lambda^2} \; \drm \lambda 
= \frac{1}{\ln(q)} \int_{g^{-1}(\varphi^{-1}(R))}^\infty \frac{f(\mu) g'(\mu)}{g(\mu)^2} \; \drm \mu \\
& = \frac{1}{\ln(q)} \int_{f(g^{-1}(\varphi^{-1}(R)))}^\infty \frac{\xi g'(f^{-1}(\xi))}{g(f^{-1}(\xi))^2 f'(f^{-1}(\xi))} \; \drm \xi  \geq \frac{1}{\ln(q)} \int_{f(g^{-1}(\varphi^{-1}(R)))}^\infty \frac{1}{g(f^{-1}(\xi))} \; \drm \xi,
\end{align*}
where we used that $\xi g'(f^{-1}(\xi))/f'(f^{-1}(\xi)) \geq g(f^{-1}(\xi))$ for all $\xi>0$ since $\xi\mapsto \xi/g(f^{-1}(\xi))$ is monotonically descreasing. Thus, \eqref{eq:integral_condition_DM} yields \eqref{eq:ass:g} in Theorem \ref{thm:spectral+diss-obs}. However, \eqref{eq:ass:g} applies for more functions. Indeed, consider \cite[Lemma 6.2]{DuyckaertsM-12} where it is shown that $\varphi(\lambda) = (\ln(\ln(\lambda)))^s \ln(\lambda)$ satisfies condition \eqref{eq:integral_condition_DM} if and only if $s>2$, whereas $\lambda \mapsto \frac{1}{\varphi(\lambda)\lambda}$ is integrable for all $s>1$. 
Finally, let us emphasize that in \cite{DuyckaertsM-12} no estimate on the constant in \eqref{eq:OBS} is obtained.
\end{remark}

\begin{example}\label{ex:compare_GST} 
Let us consider the important case of polynomial functions $f,g,h$ in the in the uncertainty principle and dissipation estimate in Theorem \ref{thm:spectral+diss-obs} and compare the observability inequality to the one obtained in \cite{GallaunST-20}. For this, consider $f(\lambda) = c_1\lambda^{\gamma_1}$, $g(\lambda) = c_2\lambda^{\gamma_2}$ for $\gamma_2 > \gamma_1>0$ and $c_1,c_2>0$, and $h(t)=t^{\gamma_3}$ for $\gamma_3>0$. To verify condition \eqref{eq:ass:g} we compute a suitable $\lambda_T \geq 0$ by
	\begin{align*}
	\int_{\lambda_T}^\infty h^{-1}\left(4\frac{\lambda}{g(f^{-1}(\lambda))}\right)\frac{1}{\lambda} \;\drm \lambda 
	= \Bigg(\frac{4 c_1^{\frac{\gamma_2}{\gamma_1}}}{c_2}\Bigg)^{\frac{1}{\gamma_3}} \int_{\lambda_T}^\infty \lambda^{-\frac{\gamma_2-\gamma_1}{\gamma_1\gamma_3}-1} \; \drm \lambda 
	= \Bigg(\frac{4 c_1^{\frac{\gamma_2}{\gamma_1}}}{c_2}\Bigg)^{\frac{1}{\gamma_3}} \frac{\gamma_1\gamma_3}{\gamma_2-\gamma_1} \lambda_T^{-\frac{\gamma_2-\gamma_1}{\gamma_1\gamma_3}}.
	\end{align*}
	Hence for $T>0$ condition \eqref{eq:ass:g} is satisfied for
	\begin{align*}
	\lambda_T = \frac{K}{\min\bigl\{1,T^{\frac{\gamma_1\gamma_3}{\gamma_2-\gamma_1}}\bigr\}}, \quad \text{with}\quad K:= \left(\frac{c_1^{\gamma_2}}{c_2^{\gamma_1}}\right)^{\frac{1}{\gamma_2-\gamma_1}} \left(4^{\frac{1}{\gamma_3}} \frac{\gamma_1\gamma_3}{\gamma_2-\gamma_1} \frac{4}{\ln(2)}\right)^{\frac{\gamma_1\gamma_3}{\gamma_2-\gamma_1}}.
	\end{align*}
	Assuming the conditions in Theorem \ref{thm:spectral+diss-obs} with functions $f,g$, and $h$ as above, we obtain the observability estimate \eqref{eq:obs} with the bound
	\begin{align}\label{eq:cost_poly}
 C_{\mathrm{obs}} \leq \frac{C_3}{T}\exp\left(\frac{6K}{\min\bigl\{1,T^{\frac{\gamma_1\gamma_3}{\gamma_2-\gamma_1}}\bigr\}} + \omega_+ T\right),
\end{align}
where $C_3$ as in Theorem \ref{thm:spectral+diss-obs}. The estimate \eqref{eq:cost_poly} shares, up to numerical constants, the same behaviour w.r.t. $c_1,c_2,C_1,C_2,T$, and $\|C\|$ as the estimate obtained in \cite{GallaunST-20}. Only the part of $K$ depending solely on $\gamma_1,\gamma_2,\gamma_3$ differs qualitatively from \cite{GallaunST-20}. 
Note that the exponential blow-up of the observability constant has to occur as $T\to 0$, e.g., for the controlled heat equation on bounded domains, see \cite{Miller-04}. This shows that the bound \eqref{eq:cost_poly} is optimal for small $T$.
\end{example}

\begin{example}\label{ex:function_g}
We give some examples of functions for which condition \eqref{eq:ass:g} in Theorem \ref{thm:spectral+diss-obs} applies. For simplicity we choose $h(t) = t$ for $t>0$ as well as $f(\lambda)=\lambda$ for $\lambda>0$ and only focus on the function $g$. Then, for given $T\in (0,1]$, assumption \eqref{eq:ass:g} reads as 
\begin{align} 
\exists \lambda_T \geq 0 : \quad \int_{\lambda_T}^\infty \frac{1}{g(\lambda)} \;\drm \lambda \le  \frac{\ln(2)}{4^2} T.
\label{eq:ass:g_ohne_h}
\end{align}
First note that for \eqref{eq:ass:g_ohne_h} the function $1/g$ has to be integrable at infinity. Hence \eqref{eq:ass:g_ohne_h} excludes for example functions $g$ such that $\lambda\mapsto g(\lambda)/\lambda$ remains bounded as $\lambda \to \infty$, as well as $g(\lambda) = \lambda \ln(1+\lambda)$ for $\lambda>0$. In the table below we give some positive examples and suitable $\lambda_T$ such that \eqref{eq:ass:g_ohne_h} holds. We abbreviate $K:= \ln(2)/4^2$ and choose $s>1$. 

\begin{center}
	{\def\arraystretch{1.5}\tabcolsep=1.25ex
		\begin{tabular}{c@{\hskip 3ex}c}
			\toprule
			\rowcolor{gray!25}
                        $g(\lambda)$ &  $\lambda_T$  
			\\[0.1cm]
			\bottomrule
                        $\exp\left(\lambda\right)$  & $\ln\left(\frac{1}{K T}\right)$   
			\\[0.1cm] \rowcolor{gray!25}
                         $\lambda^s$  & $\left(\frac{1}{(s-1)K T}\right)^{\frac{1}{s-1}} $ 
			\\
			[0.1cm]
			$\lambda\bigl(\ln(1+\lambda)\bigr)^s$  & $\exp\left(\left(\frac{1}{(s-1)K T}\right)^{\frac{1}{s-1}}\right) $ 
			\\
			[0.1cm] \rowcolor{gray!25}
			$\lambda\bigl(\ln(1+\ln(1+\lambda))\bigr)^s$  & $\exp\left(\exp\left(\left(\frac{1}{(s-1)K T}\right)^{\frac{1}{s-1}}\right)\right) $ 
			\\
			[0.1cm]
			\bottomrule
		\end{tabular}
	}
\end{center}

\end{example}

\begin{example}\label{ex:AnomalousDiffusion}
Let us also have a look at the example of anomalous diffusion on a smooth and bounded domain $\Omega \subset \R^d$, presented in \cite[Theorem 7]{DuyckaertsM-12}. Consider the Dirichlet Laplacian $\Delta_{\mathrm{Dir}}$ in $L_2(\Omega)$ observed on an open non-empty region $\thickset \subset \Omega$. Let $(\varphi_k)_{k\in \N}$ be the orthonormal basis of $L_2(\Omega)$ consisting of eigenvectors of $\Delta_{\mathrm{Dir}}$ with corresponding eigenvalues $(-\lambda_k)_{k\in \N}$. Based on the local elliptic Carleman estimates obtained in \cite{LebeauR-95}, in \cite[Theorem 3]{LebeauZ-98} and \cite[Theorem 14.6]{JerisonL-99} the following estimate for sums of eigenfunctions of $\Delta_{\mathrm{Dir}}$ was proven: 
\begin{equation}\label{eq:spectral_ineq_AD}
\forall \lambda>0,\,x\in L_2(\Omega): \quad \|P_\lambda x\|_{L_2(\Omega)} \le C_1 \euler^{c \sqrt{\lambda}} \|\1_\thickset P_\lambda x\|_{L_2(\thickset)},
\end{equation}
with uniform $C_1,c>0$, $x=\sum_{k\in \N} a_k \varphi_k$, and $P_\lambda x := \sum_{\lambda_k\le \lambda} a_k \varphi_k$ being the projector on the spectral subspace of $\Delta_{\mathrm{Dir}}$ below $\lambda$. 
For $s>1$ let $g\from (0,\infty) \to (0,\infty)$ be any of the functions given in the table of Example \ref{ex:function_g} and consider the anomalous diffusion operator $-g(\sqrt{-\Delta_{\mathrm{Dir}}})$ defined by the functional calculus of self-adjoint operators. This yields that the $C_0$-semigroup $(S(t))_{t\geq 0}$ generated by $-g(\sqrt{-\Delta_{\mathrm{Dir}}})$ can be represented as
\[
S(t) x = \sum_{k\in \N} a_k \euler^{-g(\sqrt{\lambda_k}) t} \varphi_k
\]
with $x=\sum_{k\in \N} a_k \varphi_k$, and hence it satisfies the dissipation estimate 
\begin{equation}\label{eq:diss_estimate_AD}
\|(\id-P_\lambda)S(t)x\|_{L_2(\Omega)} = \Big\| \sum_{\lambda_k> \lambda} a_k \euler^{-g(\sqrt{\lambda_k}) t} \varphi_k\Big\|_{L_2(\Omega)} \le \euler^{-g(\sqrt{\lambda}) t} \|x\|_{L_2(\Omega)},
\end{equation}
for all $\lambda>0$, $t>0$, and $x=\sum_{k\in \N} a_k \varphi_k \in L_2(\Omega)$. Therefore, the assumptions of Theorem \ref{thm:spectral+diss-obs} are satisfied with $C\in\cL(L_2(\Omega),L_2(\thickset))$ being the restriction operator, and $f(\lambda)=c\sqrt{\lambda}$ for $\lambda>0$ and $h(t)=t$ for $t >0$. Theorem \ref{thm:spectral+diss-obs} not only gives an observability estimate for the semigroup associated to $-g(\sqrt{-\Delta_{\mathrm{Dir}}})$, but also an estimate on the observation constant (suitable values of $\lambda_T$ are given in the table of Example \ref{ex:function_g}). Note that we improve the result given in \cite[Theorem 7]{DuyckaertsM-12}, since, e.g., the function $g(\lambda)=\lambda\bigl(\ln(1+\ln(1+\lambda))\bigr)^s$ for $s\in (1,2]$ does not satisfy the integrability condition in \cite{DuyckaertsM-12}. Finally, let us emphasize that, using duality results, the observability inequality given in Theorem \ref{thm:spectral+diss-obs} implies null-controllability of the anomalous diffusion equation with control set $\thickset$.
\end{example}

We now complete this section by giving the proof of Theorem \ref{thm:spectral+diss-obs}. Instead of the telescopic series method used in \cite{DuyckaertsM-12}, the proof of Theorem \ref{thm:spectral+diss-obs} is based on an iteration argument due to \cite{TenenbaumT-11, NakicTTV-20, GallaunST-20}. 

\begin{proof}[Proof of Theorem \ref{thm:spectral+diss-obs}.]
W.l.o.g., we may assume that $f(\lambda) = \lambda$ for all $\lambda>0$. Indeed, for a general function $f:(0,\infty) \to (0,\infty)$, satisfying the assumptions of Theorem \ref{thm:spectral+diss-obs}, the claim then follows by replacing $g$ by $g\circ f^{-1}$ and considering the rescaled operators $\widetilde P_\lambda := P_{f(\lambda)}$, $\lambda >0$, in \eqref{eq:ass:uncertainty} and \eqref{eq:ass:dissipation}.
Further, w.l.o.g. we may assume that the semigroup $\bigl(S(t)\bigr)_{t\geq 0}$ 
is bounded, i.e., $\omega=0$. Indeed, the rescaled semigroup $\widetilde{S}(t) := \euler^{-\omega t} S(t)$ ($t\geq 0$) is bounded and applying the theorem for bounded semigroups yields
 \[
\forall x\in X: \quad \|\widetilde{S}(T)x\|_X\leq C_{\mathrm{obs}} \int_0^T \lVert C \widetilde{S}(\tau) x \rVert_{Y} \;\drm \tau,
 \]
and therefore
\begin{equation*}
\|S(T)x\|_X = \euler^{\omega T} \|\widetilde{S}(T)x\|_X 
\leq \euler^{\omega T} C_{\mathrm{obs}} \int_0^T \euler^{-\omega \tau} \lVert C S(\tau) x \rVert_{Y} \;\drm \tau \leq \euler^{\omega_+ T} C_{\mathrm{obs}} \int_0^T \lVert C S(\tau) x \rVert_{Y} \;\drm \tau.
\end{equation*}
Thus, it suffices to prove the theorem for bounded semigroups.
The proof is divided in two steps.

\textit{Step 1. Approximate observability estimate.} Let $\lambda >0$, $t \in (0,T]$, and $x\in X$. By \eqref{eq:ass:uncertainty} and \eqref{eq:ass:dissipation}, for $\tau\in (0,t]$ we obtain
\begin{align}
 \lVert S(\tau) x \rVert_X & \leq \lVert P_\lambda S(\tau) x \rVert_X + \lVert (\id-P_\lambda)S(\tau) x \rVert_X \nonumber\\
&\leq  C_1 \euler^\lambda \lVert CP_\lambda S(\tau) x \rVert_X + \lVert (\id-P_\lambda)S(\tau) x \rVert_X \nonumber \\
& \leq  C_1 \euler^\lambda \lVert CS(\tau) x \rVert_X + C_1 \euler^\lambda \lVert C \rVert_{\cL (X,Y)} \lVert(\id-P_\lambda)S(\tau) x \rVert_X  + \lVert(\id-P_\lambda)S(\tau) x \rVert_X  \nonumber \\
& \leq  C_1 \euler^\lambda\lVert CS(\tau) x \rVert_X  + \left(C_1 \euler^\lambda\lVert C \rVert_{\cL (X,Y)}+1\right) C_2 \euler^{-g(\lambda)h(\tau)} \lVert x \rVert_X \nonumber \\
& \leq  C_1 \euler^\lambda\lVert CS(\tau) x \rVert_X  + \left(C_1 \lVert C \rVert_{\cL (X,Y)}+1\right) C_2 \euler^{-g(\lambda)h(\tau)+\lambda} \lVert x \rVert_X.\label{eq:to_be_integrated}
\end{align}
Since $\bigl(S(t)\bigr)_{t\geq 0}$ is a bounded semigroup, we get
$$\lVert S(t)x \rVert_X = \lVert S(t-\tau) S(\tau) x\lVert_X  \leq M \lVert S(\tau) x \rVert_X.$$
Let $\delta\in (0,1)$. Integrating \eqref{eq:to_be_integrated} with respect to $\tau \in [\delta t,t]$, we obtain
\begin{align*}
& \frac{(1-\delta)t}{M}\lVert S(t)x \rVert_X \\
&\leq C_1 \euler^\lambda \int_{\delta t}^t \lVert CS(\tau) x \rVert_Y\;\drm \tau + \left(C_1 \lVert C \rVert_{\cL (X,Y)}+1\right) C_2 \int_{\delta t}^t \euler^{-g(\lambda)h(\tau) + \lambda} \;\drm \tau \,\lVert x \rVert_X \\
&\leq C_1 \euler^\lambda \int_{\delta t}^t \lVert CS(\tau) x \rVert_Y \;\drm \tau + \left(C_1 \lVert C \rVert_{\cL (X,Y)}+1\right) C_2 (1-\delta)t \,\euler^{-g(\lambda)h(\delta t)+\lambda} \,\lVert x \rVert_X.
\end{align*}
By multiplying both sides by $M/(1-\delta)t$ we finally arrive at
\begin{equation} \label{eq:approx_obs}
\norm{S(t)x}_X\leq C_{\mathrm{obs}}(t,\lambda) \int_{\delta t}^t \lVert C S(\tau) x \rVert_{Y} \drm \tau + \alpha(t,\lambda) \|x\|_X
\end{equation}
with
\begin{align}\label{eq:bound_Cobs_alpha} 
C_{\mathrm{obs}}(t,\lambda) = \frac{M}{(1-\delta)t} C_1 \euler^{\lambda} \quad \text{and} \quad \alpha(t,\lambda) = M \left(C_1 \lVert C \rVert_{\cL (X,Y)}+1\right) C_2 \euler^{-g(\lambda)h(\delta t) + \lambda}.
\end{align}

\textit{Step 2. Iteration argument.} Set $K := 2^{\frac{\euler}{2}} M \left(C_1\lVert C \rVert_{\cL (X,Y)}+1\right)C_2$ and let $\lambda_T\geq 0$ satisfy \eqref{eq:ass:g}. We define
\begin{equation*}
\lambda_0 := \frac{2 \ln (K)}{\euler \ln (2)} + 2\lambda_T
\label{eq:def_lambda0}
\end{equation*}
and for $k\in \N$ we set $\lambda_k := \lambda_0 2^k$. Furthermore, for $k\in \N_0$, we define
\begin{equation*}
\tau_k:= \frac{T}{2} \euler^{-\lambda_k} + 2\max\{1,T\}h^{-1}\left(4\frac{\lambda_k}{g(\lambda_k)}\right)>0
\label{eq:def_tau_k}
\end{equation*}
and recursively $T_{k+1} := T_{k} - \tau_k$, where $T_0:=T$. Next, we show that $\tau_k, T_k \in (0,T]$ for all $k\in \N_0$. For this, note that since $C_2\geq 1$ by assumption, we have $K\geq 2^{\frac{\euler}{2}}$ and hence $\ln(K) \geq \frac{\euler \ln(2)}{2}$. This shows that $\lambda_0 \geq 1$ and hence
\begin{equation}
\sum_{k=0}^\infty \euler^{-\lambda_k} \le \sum_{k=0}^\infty \euler^{-2^k} \le 1.
\label{eq:sum_lambdak} 
\end{equation}
Since $\lambda\mapsto \lambda/g(\lambda)$ is monotonically decreasing and $h^{-1}$ is monotonically increasing, for all $k\in \N_0$ we have
\begin{align*}
h^{-1}\left(4\frac{\lambda_k}{g(\lambda_k)}\right) \le \int_{k-1}^{k} h^{-1}\left(4\frac{\lambda_0 2^s}{g(\lambda_0 2^s)}\right) \;\drm s,
\end{align*}
and therefore
\begin{align*}
\sum_{k=0}^\infty \tau_k &\le \frac{T}{2} \sum_{k=0}^\infty \euler^{-\lambda_k} + 2\max\{1,T\} \sum_{k=0}^\infty \; \int_{k-1}^{k} h^{-1}\left(4\frac{\lambda_0 2^s}{g(\lambda_0 2^s)}\right) \; \drm s \nonumber \\
& \leq \frac{T}{2} + 2\max\{1,T\} \int_{-1}^{\infty} h^{-1}\left(4\frac{\lambda_0 2^s}{g(\lambda_0 2^s)}\right) \; \drm s \nonumber \\
& = \frac{T}{2} + \frac{2\max\{1,T\}}{\ln(2)} \int_{\lambda_0/2}^\infty h^{-1}\left(4\frac{\lambda}{g(\lambda)}\right)\frac{1}{\lambda} \;\drm \lambda. 
\end{align*}
Since $\lambda_0/2 \geq \lambda_T$, by \eqref{eq:ass:g} we observe
\begin{align*}
\frac{2\max\{1,T\}}{\ln(2)} \int_{\lambda_0/2}^\infty h^{-1}\left(4\frac{\lambda}{g(\lambda)}\right)\frac{1}{\lambda} \;\drm \lambda \le \frac{2\max\{1,T\}}{\ln(2)} \frac{\ln(2)\min\{1,T\}}{4}  = \frac{T}{2}.
\end{align*}
Hence, we have $\sum_{k=0}^\infty \tau_k \le T$, which gives $T_k = T - \sum_{l=0}^{k-1} \tau_k \in (0,T]$.

Let $x\in X$. For $k\in\N_0$, we apply the approximate observability estimate \eqref{eq:approx_obs} to $t=\tau_k$, $\delta=\frac{1}{2}$, $\lambda=\lambda_k$ and $S(T_{k+1})x$ to obtain
\begin{align}
\norm{S(T_k)x}_X\leq C_{\mathrm{obs}}(\tau_k,\lambda_k) \int_{0}^{T} \lVert C S(\tau) x \rVert_{Y} \drm \tau + \alpha(\tau_k,\lambda_k) \|S(T_{k+1})x\|_X
\label{eq:to_be_iterated} 
\end{align}
with 
\begin{align*} 
C_{\mathrm{obs}}(\tau_k,\lambda_k) = \frac{2MC_1}{\tau_k}\euler^{\lambda_k}\quad \text{and}\quad \alpha(\tau_k,\lambda_k) = K \euler^{-g(\lambda_k) h(\tau_k/2) + \lambda_k}.
\end{align*}
Note that for technical reasons the constant $K$ in $\alpha(\tau_k,\lambda_k)$ is by the factor $2^{\frac{\euler}{2}}\geq 1$ larger compared to \eqref{eq:bound_Cobs_alpha}, and we performed integration on $(0,T]$ instead of $(\tau_k/2,\tau_k]$.
The inequalities \eqref{eq:to_be_iterated} can be iterated. Starting from $k=0$, after $N+1$ steps we obtain
\begin{align} 
\|S(T)x\|_X \leq \left(C_{\mathrm{obs}}(\tau_0,\lambda_0) + \sum_{k=1}^N C_{\mathrm{obs}}(\tau_k,\lambda_k) \prod_{l = 0}^{k-1} \alpha(\tau_l,\lambda_l) \right) \int_0^T \lVert C S(\tau) x \rVert_{Y} \drm \tau \label{eq:after_iteration1}\\
+ \|S(T_{N+1})x\|_X \prod_{k=0}^N \alpha(\tau_k,\lambda_k). \label{eq:after_iteration2}
\end{align}

Let us first consider the term \eqref{eq:after_iteration2}. For $k\in\N_0$, by definition of $\tau_k$ and since $h$ is monotonically increasing, we obtain
$$h(\tau_k/2) \geq h\left(\max\{1,T\} h^{-1}\left(4\frac{\lambda_k}{g(\lambda_k)}\right)\right)\geq 4\frac{\lambda_k}{g(\lambda_k)},$$
and therefore
\begin{align*}
\alpha(\tau_k,\lambda_k) \le K \euler^{-g(\lambda_k) \frac{4\lambda_k}{g(\lambda_k)} + \lambda_k} \le K \euler^{-3\lambda_k}.
\end{align*}
Using that 
$$\sum_{k=0}^N \lambda_k = \lambda_0 \sum_{k=0}^N 2^k = \lambda_0 \left(2^{N+1}-1\right) = \lambda_{N+1} - \lambda_0,$$
we obtain
\begin{align*}
\prod_{k=0}^N \alpha(\tau_k,\lambda_k) 
\le K^{N+1} \euler^{-3\sum_{k=0}^N \lambda_k} 
= K^{N+1} \euler^{-3\lambda_{N+1} + 3\lambda_0} \to 0 \quad(N\to\infty),
\end{align*}
since $\lambda_{N+1} \geq 2^{N+1}$ for all $N\in\N_0$. Since $\bigl(S(t)\bigr)_{t\geq0}$ is bounded, the summand \eqref{eq:after_iteration2} converges to zero as $N \to \infty$.

It remains to estimate the right-hand side in \eqref{eq:after_iteration1}. First, note that for $k\in\N_0$ we have, by the definition of $\tau_k$, that $\tau_k \geq \frac{T}{2} \euler^{-\lambda_k}$ and hence
\begin{align*}
C_{\mathrm{obs}}(\tau_k,\lambda_k) = \frac{2MC_1}{\tau_k}\euler^{\lambda_k} \le \frac{4MC_1}{T}\euler^{2\lambda_k}.
\end{align*}
Therefore,
\begin{align*} 
\sum_{k=1}^N C_{\mathrm{obs}}(\tau_k,\lambda_k) \prod_{l = 0}^{k-1} \alpha(\tau_l,\lambda_l) 
&\le \frac{4MC_1}{T} \sum_{k=1}^N \euler^{2\lambda_k} K^{k} \euler^{-3\lambda_{k} + 3\lambda_0}
=\frac{4MC_1}{T} \euler^{3\lambda_0} \sum_{k=1}^N K^{k} \euler^{-\lambda_k}. \label{eq:sum_of_Cobs_k}
\end{align*}
Since $\lambda_k = \lambda_0 2^k$ for $k\in\N_0$, the series converges for $N\to \infty$. Next, we show that 
\begin{equation*}\label{eq:claim_estimate_sum}
\sum_{k=1}^\infty K^{k} \euler^{-\lambda_k} \le 1.
\end{equation*}
For this, we estimate
\begin{equation*}
\sum_{k=1}^\infty K^k \euler^{-\lambda_k} \le \sup_{x\geq 1} K^x \euler^{-\frac{\lambda_0}{2}2^x} \sum_{k=1}^\infty \euler^{-\frac{\lambda_k}{2}} = \left( \frac{2 \ln (K)}{\lambda_0 \euler \ln (2)} \right)^{\frac{\ln (K)}{\ln (2)}} \sum_{k=1}^\infty \euler^{-\frac{\lambda_k}{2}},
\end{equation*}
where the last identity follows from elementary calculus. Since $\lambda_0 \geq \frac{2\ln(K)}{\euler \ln(2)}$ and by \eqref{eq:sum_lambdak} we have 
\begin{equation*}
\left( \frac{2 \ln (K)}{\lambda_0 \euler \ln (2)} \right)^{\frac{\ln (K)}{\ln (2)}} \sum_{k=1}^\infty \euler^{-\frac{\lambda_k}{2}} \le \sum_{k=0}^\infty \euler^{-\lambda_k} \le 1.
\end{equation*}
We thus obtain
\begin{align*} 
\sum_{k=1}^\infty C_{\mathrm{obs}}(\tau_k,\lambda_k) \prod_{l = 0}^{k-1} \alpha(\tau_l,\lambda_l) \le \frac{4MC_1}{T}  \euler^{3\lambda_0}.
\end{align*}
By this estimate and $C_{\mathrm{obs}}(\tau_0,\lambda_0) \le \frac{4MC_1}{T} \euler^{2\lambda_0} \le \frac{4MC_1}{T} \euler^{3\lambda_0}$, \eqref{eq:after_iteration1} yields
$$\norm{S(T)x}_X\leq C_{\mathrm{obs}} \int_0^T \lVert C S(\tau) x \rVert_{Y} \;\drm \tau \quad \text{with} \quad 
C_{\mathrm{obs}} \le \frac{8MC_1}{T} \euler^{3\lambda_0}.$$
By definition of $\lambda_0$, we conclude
\begin{align*}
C_{\mathrm{obs}} &\le \frac{8MC_1}{T}  \exp \left(\frac{12 \ln (K)}{\euler \ln (2)} + 6\lambda_T\right) \le \frac{8MC_1}{T} K^{\frac{6}{\euler \ln (2)}} \euler^{6\lambda_T}.
\end{align*}
The claim now follows by inserting the definition of $K$.
\end{proof}

\section{Final State Observability for Subordinated Semigroups}
\label{sec:subordination}

We now apply Theorem \ref{thm:spectral+diss-obs} to so-called subordinated semigroups introduced by Bochner in 1949 \cite{Bochner-49}. We will only deal with strongly continuous semigroups in the presentation; however, note that subordination can also be performed in more general settings (see e.g.\ \cite{KruseMS-21}), in particular for dual semigroups of strongly continuous semigroups.

\begin{definition}[Bernstein function]
	Let $\varphi\from(0,\infty)\to[0,\infty)$. 
	Then $\varphi$ is called \emph{Bernstein function} provided $\varphi \in C^\infty(0,\infty)$, and $(-1)^{k-1} \varphi^{(k)} \geq 0$ for all $k\in\N$.
\end{definition}

Bernstein functions can be characterised as follows.

\begin{proposition}[{\cite[Theorem 3.2]{Schilling-12}}]
  \label{prop:levy_khintchine}
  Let $\varphi\from(0,\infty)\to [0,\infty)$. The following are equivalent.
  \begin{enumerate}
    \item
      $\varphi$ is a Bernstein function.
    \item
      There exist unique constants $a,b\geq 0$ and a unique positive Radon measure $\mu$ on $(0,\infty)$ satisfying $\int_{(0,\infty)} 1\wedge t\,\mu(\d t) <\infty$ such that
      \begin{equation}
	\varphi(\lambda) = a + b \lambda + \int\limits_{(0,\infty)} \bigl(1-e^{-\lambda t}\bigr)\,\mu(\d t) \quad (\lambda > 0) \label{eq:representation}.
      \end{equation}
  \end{enumerate}
\end{proposition}
The representation of a Bernstein function $\varphi$ in \eqref{eq:representation} is called \emph{L\'{e}vy--Khinchin representation} and the triplet $(a,b,\mu)$ is called \emph{L\'{e}vy triplet} of $\varphi$. \

\begin{definition}
 \label{def:convolutionsemigroup}
  Let $(\mu_t)_{t \geq 0}$ be a family of Radon measures on $[0,\infty)$ and $\mu$ a Radon measure on $[0,\infty)$. 
  Then $(\mu_t)_{t \geq 0}$ is called 
   \begin{enumerate}
     \item
      a family of \emph{sub-probability measures} if $\forall t \in [0, \infty): \quad \mu_t \big( [0, \infty) \big) \leq 1$,
     
     \item
      a \emph{convolution semigroup} if $\mu_0 =\delta_0$ and $\forall s,t \in [0, \infty): \quad \mu_t \ast \mu_s = \mu_{t+s}$,  
     
     \item
      \emph{vaguely continuous at $s \in [0, \infty)$} with limit $\mu$ if 
      \begin{equation}
 	   \forall f \in C_{\mathrm{c}}[0, \infty): \quad \lim\limits_{t \rightarrow s }\int\limits_{[0, \infty)} f(\lambda) \,\mu_t (\d \lambda) =  
 	   \int\limits_{[0, \infty)} f(\lambda) \,\mu (\d \lambda). 
 	   \label{eq:VagueContinuity}
       \end{equation}
   \end{enumerate}	 	
\end{definition} 
Note that if $(\mu_t)_{t\geq 0}$ is a family of sub-probability measures which is vaguely continuous at $0$ with limit $\delta_0$, then $(\mu_t)_{t\geq 0}$ is also weakly continuous, i.e., \eqref{eq:VagueContinuity} actually even holds for all $f \in C_{\mathrm{b}}[0, \infty)$. 

The Laplace transform relates Bernstein functions to vaguely continuous convolution semigroups of sub-probability measures on $[0,\infty)$.

\begin{proposition}[{\cite[Theorem 5.2]{Schilling-12}}]
\label{prop:Laplace-Transform}
  Let $(\mu_t)_{t \geq 0}$ be a convolution semigroup of sub-pro\-ba\-bi\-li\-ty measures on $[0, \infty)$ 
  which is vaguely continuous at $0$ with limit $\delta_0$. 
  Then there exists a unique Bernstein function $\varphi\from(0,\infty)\to [0,\infty)$ such that for all $t\geq 0$ the Laplace transform of $\mu_t$ is given by 
  \[
    \mathcal{L} (\mu_t) = \euler^{-t\varphi}. 
  \]
  Conversely, given any Bernstein function $\varphi\from (0,\infty)\to [0,\infty)$, there exists a unique vaguely continuous convolution semigroup $(\mu_t)_{t\geq 0}$ of sub-probability measures on $[0, \infty)$ such that the above equation holds. 
\end{proposition}

\begin{definition}
    Let $X$ be a Banach space, $(S(t))_{t\geq 0}$ a bounded $C_0$-semigroup on $X$, and $\varphi\from (0,\infty)\to [0,\infty)$ a Bernstein function. Let $(\mu_t)_{t\geq 0}$ the vaguely continuous convolution semigroup of sub-probability measures on $[0,\infty)$ associated with $\varphi$. For $t\geq 0$ we define $S^\varphi(t)\in \calL(X)$ by 
    \[S^\varphi(t) x:= \int_{[0,\infty)} S(s) x \; \mu_t(\d s) \quad(x\in X).\]
    Then $(S^\varphi(t))_{t\geq 0}$ is called \emph{subordinated semigroup} to $(S(t))_{t\geq 0}$ w.r.t.\ $\varphi$.
\end{definition}

A first question to ask is whether this defines again a $C_0$-semigroup. 
The answer is affirmative. 

\begin{proposition}[{see, e.g., \cite[Proposition 13.1]{Schilling-12}}]
\label{prop:subordination}
Let $X$ be a Banach space, $(S(t))_{t\geq 0}$ a bounded $C_0$-semigroup on $X$, $\varphi\from (0,\infty)\to [0,\infty)$ a Bernstein function, and $(S^\varphi(t))_{t\geq 0}$ the subordinated semigroup to $(S(t))_{t\geq 0}$ w.r.t.\ $\varphi$. 
Then $(S^\varphi(t))_{t\geq 0}$ is again a bounded $C_0$-semigroup on $X$.
\end{proposition}

In view of Proposition \ref{prop:subordination}, the generator of $(S^\varphi(t))_{t\geq 0}$ is given by $-\varphi(-A)$, where $A$ is the generator of $(S(t))_{t\geq 0}$; cf. \cite[Theorem 4.3]{Phillips1952}.

The next lemma will be an important ingredient in order to apply \cref{thm:spectral+diss-obs} to subordinated semigroups. 
It describes how dissipation of the original semigroup transfers to subordinated semigroups. 

\begin{lemma}
\label{prop:dissipation_subordination}
    Let $X$ be a Banach space, $(P_\lambda)_{\lambda> 0}$ in $\cL(X)$, and $(S(t))_{t\geq 0}$ a bounded $C_0$-semigroup on $X$. 
    Let further $T>0$, $g\from (0,\infty) \to (0,\infty)$ be measurable and $C_2\geq 1$ such that
    \begin{align*}
    \forall \lambda>0,\, t\in (0,T],\, x\in X: \quad \lVert (\id-P_\lambda) S(t) x \rVert_{X} \le C_2 \euler^{-g(\lambda)t} \lVert x \rVert_{X}.
    \end{align*}
    Let finally $\varphi\from (0,\infty)\to [0,\infty)$ be a Bernstein function and $(S^\varphi(t))_{t\geq 0}$ the subordinated semigroup to $(S(t))_{t\geq 0}$ w.r.t.\ $\varphi$. 
    Then
    \begin{align*}
    \forall \lambda>0,\, t\in (0,T],\, x\in X: \quad \lVert (\id-P_\lambda) S^\varphi(t) x \rVert_{X} \le C_2 \euler^{-\varphi(g(\lambda))t} \lVert x \rVert_{X}.
    \end{align*}
\end{lemma}

\begin{proof}
    Let $(\mu_t)_{t\geq 0}$ the vaguely continuous convolution semigroup associated with $\varphi$.
    Let $\lambda>0$, $t\in (0,T]$, and $x\in X$.
    Then
    \begin{align*}
        \lVert (\id-P_\lambda) S^\varphi(t) x \rVert_{X} & = \lVert (\id-P_\lambda) \int_0^\infty S(s)x \;\mu_t(\d s) \rVert_{X} \leq \int_0^{\infty} \lVert (\id-P_\lambda) S(s) x\rVert_{X} \;\mu_t(\d s)\\
        & \leq C_2 \int_0^{\infty} \euler^{-g(\lambda) s} \;\mu_t(\d s) \norm{x}_X.
    \end{align*}
    Note that the $C_0$-semigroup $(\euler^{-g(\lambda) t})_{t\geq 0}$ on $\mathds{K}$ has the generator $-g(\lambda)$. Thus, the corresponding subordinated semigroup has the generator $-\varphi(g(\lambda))$ by Proposition \ref{prop:subordination}.
    Hence, $\int_0^{\infty} \euler^{-g(\lambda) s} \,\mu_t(\d s) = \euler^{-\varphi(g(\lambda)) t}$ and thus   
    \begin{equation*}
        \lVert (\id-P_\lambda) S^\varphi(t) x \rVert_{X} \leq C_2  \euler^{-\varphi(g(\lambda)) t} \norm{x}_X. \hfill \qedhere
    \end{equation*}    
\end{proof}

\begin{theorem}
\label{thm:obs_subordination}
    Let $X$ and $Y$ be Banach spaces, $C\in \cL(X,Y)$, $(P_\lambda)_{\lambda>0}$ in $\cL(X)$, and $(S(t))_{t\geq 0}$ a bounded $C_0$-semigroup on $X$. Let $T>0$, $g\from (0,\infty) \to (0,\infty)$ measurable and $C_1\geq 0$, $C_2\geq 1$ such that
    \begin{align*} 
    \forall \lambda>0,\, x\in X: \quad \lVert P_\lambda x \rVert_{ X } \le C_1 \euler^{\lambda} \lVert C  P_\lambda x \rVert_{Y } 
    \end{align*}
    and
    \begin{align*}
    \forall \lambda>0,\, t\in (0,T],\, x\in X: \quad \lVert (\id-P_\lambda) S(t) x \rVert_{X} \le C_2 \euler^{-g(\lambda)t} \lVert x \rVert_{X}.
    \end{align*}
    Let $\varphi\from (0,\infty)\to [0,\infty)$ be a Bernstein function, $(S^\varphi(t))_{t\geq 0}$ the subordinated semigroup to $(S(t))_{t\geq 0}$ w.r.t.\ $\varphi$, and assume that $\lambda \mapsto \lambda/(\varphi\circ g(\lambda))$ is monotonically decreasing and that $1/(\varphi\circ g)$ integrable at $\infty$.
    Then there exists $C_\mathrm{obs}\geq 0$ such that
    \begin{equation} \label{eq:obs:subord}
    \forall x\in X: \quad \norm{S^\varphi(T)x}_X\leq C_{\mathrm{obs}} \int_0^T \lVert C S^\varphi(\tau) x \rVert_{Y} \;\drm \tau.
    \end{equation}
\end{theorem}

\begin{proof}
    By Lemma \ref{prop:dissipation_subordination} the dissipation estimate for $(S^\varphi(t))_{t\geq 0}$ is satisfied. Thus, the statement follows from \cref{thm:spectral+diss-obs}.
\end{proof}

\begin{example}
   As a standard example, we consider the heat semigroup on $L_p(\R^n)$ where $p\in [1,\infty)$.
   Let $A:=\Delta$ the Laplacian in $L_p(\R^n)$ with domain $W_p^2(\R^n)$ and $(S(t))_{t\geq 0}$ the $C_0$-semigroup generated by $A$.
   Let $g(\lambda):=\lambda^2$ for $\lambda>0$, $s\in (\frac{1}{2},1)$, and $\varphi(\lambda):=\lambda^s$ for $\lambda>0$ which defines a Bernstein function.
	Then $\lambda \mapsto \lambda/\varphi(g(\lambda))$ is monotonically decreasing and $1/(\varphi\circ g)$ is integrable at $\infty$.   
   Let $E\subseteq\R^n$ be a so-called thick set, i.e., $\thickset$ is measurable and
   \[\left\lvert \thickset \cap \left( \bigtimes_{i=1}^n (0,L_i) + x \right) \right\rvert \geq \rho \prod_{i=1}^n L_i \quad(x\in\R^n),\]
   for some $\rho\in (0,1]$ and $L\in (0,\infty)^n$. Here, $|\cdot|$ denotes Lebesgue measure in $\R^n$. Moreover, let  $C\in\calL(L_p(\R^n),L_p(\thickset))$ be the restriction operator and for $\lambda>0$ let $P_\lambda\in \calL(L_p(\R^n))$ be induced by a smooth cut-off function $\1_{B(0,\lambda/2)} \leq \chi_\lambda \leq \1_{B(0,\lambda)}$ in Fourier space; cf. \cite{BombachGST-20}.
   Then the uncertainty principle is satisfied by  a Logvinenko--Sereda type theorem (see \cref{prop:Logvinenko-Sereda_Rd} below or \cite{EgidiV-18, WangWZZ-19, BombachGST-20}). Moreover, exploiting the heat kernel, we deduce a dissipation estimate with $g$ and $h(t) = t$ for all $t>0$ (see \cite{BombachGST-20}).
   Hence, \cref{thm:obs_subordination} is applicable and we obtain final state observability for $(S^\varphi(t))_{t\geq 0}$, which is generated by the fractional Laplacian $-\varphi(-A) = -(-A)^s = -(-\Delta)^s$.
   Note that our method does not work for $s\in (0,\frac{1}{2}]$, as $1/(\varphi\circ g)$ is not integrable at $\infty$ anymore. In \cite{Koenig2020}, it was shown that the $C_0$-semigroup generated by $-(-\Delta)^{1/2}$ on $L_2(\R)$ does not satisfy a final state observability estimate. 
   \label{ex:FracLaplace}
\end{example}

\begin{remark}
In order to apply \cref{thm:obs_subordination} a rather explicit knowledge about the Bernstein function $\varphi$ is needed. 
This can be difficult if only the L{\'e}vy triplet $(a,b,\mu)$ of $\varphi$ is given. We only consider the case $a=0$.

If $b \neq 0$, nothing has to be done since this means $\varphi(\lambda) \geq b \lambda$ for all $\lambda>0$. 

If $b = 0$, however, a closer examination of $\mu$ becomes necessary.
It is easy to see that for all $\lambda>0$ we have
 \[
  \frac{1}{2} \lambda \int\limits_{(0,\frac{1}{\lambda})} t \, \mu (\d t) \leq \varphi(\lambda) \leq \lambda \int\limits_{(0,\frac{1}{\lambda})} t \, \mu (\d t) + 2 \mu \Bigl( \bigl[\frac{1}{\lambda}, \infty \bigr) \Bigr), 
 \]
since $\frac{1}{2} x \leq 1 - \e^{-x} \leq x$ for $x\in (0,1)$ (note that $\lambda t\in (0,1)$ for $t\in (0,\frac{1}{\lambda})$).
Thus, growth properties of $\varphi$ can be studied by looking at $\lambda \mapsto \lambda \int_{(0,\frac{1}{\lambda})} t \, \mu (\d t)$ instead.
 As a (standard) example,
let $s \in (0,1)$ and $\varphi\from(0,\infty)\to [0,\infty)$ be defined by the L{\'e}vy triplet $(0,0,\mu)$ with $\mu(B) := \int_B t^{-1-s}\,\d t$ for $B\subseteq (0,\infty)$ measurable, i.e.,
 \begin{equation}\label{eq:example_phi}
  \varphi(\lambda) = \int\limits_0^{\infty} \bigl(1 - \e^{-\lambda t} \bigr) t^{-1-s} \d t \quad(\lambda>0). 
 \end{equation}
 Since
 \[
  \varphi(\lambda) \geq \frac{1}{2}\lambda \int\limits_{(0,\frac{1}{\lambda})} t \, \mu (\d t) =  \frac{\lambda^{s}}{2(1-s)} \quad(\lambda>0),
 \]
  we can apply \cref{thm:obs_subordination} for the Laplacian on $L_p(\R^n)$ as in \cref{ex:FracLaplace}, and for $\varphi$ as in \eqref{eq:example_phi} with $s \in (\frac{1}{2},1)$.
\end{remark}

\section{Final State Observability for semigroups of sub-probability measures}
\label{sec:application}

In this section we apply our results to semigroups generated by families $(\mu_t)_{t \geq 0}$ of sub-probability measures in $\R^n$. The following definition is analogous to \cref{def:convolutionsemigroup} for Radon measures on $[0,\infty)$.

\begin{definition}
 \label{def:convolutionsemigroupRn}
  Let $(\mu_t)_{t \geq 0}$ be a family of Radon measures on $\R^n$ and $\mu$ a Radon measure on $\R^n$. 
  Then $(\mu_t)_{t \geq 0}$ is called 
   \begin{enumerate}
     \item
      a family of \emph{sub-probability measures} if $\forall t \in [0,\infty): \quad \mu_t (\R^n) \leq 1$,
     \item
      a \emph{convolution semigroup} if $\mu_0 =\delta_0$ and $\forall s,t \in [0,\infty): \quad \mu_t \ast \mu_s = \mu_{t+s}$,  
     \item
      \emph{vaguely continuous at $s \in [0,\infty)$} with limit $\mu$ if 
      \[
 	   \forall f \in C_{\mathrm{c}}(\R^n): \quad \lim\limits_{t \rightarrow s }\int\limits_{\R^n} f(x) \,\mu_t (\d x) =  
 	   \int\limits_{\R^n} f(x) \,\mu (\d x). 
       \] 
   \end{enumerate}	 	
\end{definition} 

Due to the following theorem, the Fourier transform of vaguely continuous semigroups of sub-probability measures on $\R^n$ corresponds to continuous positive-definite functions. 
Recall that a function $f\from \R^n \to \C$ is called \emph{positive-definite} if for all $k\in\N$ and all $x_1, \dots, x_k$ the matrix $\bigl( f(x_i - x_l) \bigr)_{1 \leq i,l \leq k} \in \C^{k \times k}$ is positive semi-definite in the spectral sense, i.e., all its eigenvalues are non-negative. 

\begin{theorem}[{Bochner, \cite[Theorem 3.5.7]{Jacob-01}}]
 Let $\mu$ be a positive and finite Radon measure on $\R^n$. 
 Then its \emph{Fourier transform}
 \[
  \fourier \mu\from \R^n \ni \xi \mapsto (\fourier \mu)(\xi) := \frac{1}{(2\pi)^{\frac{n}{2}}} \int\limits_{\R^n} \e^{-i \ip{x}{\xi}} \mu (\d x)
 \]
 is a continuous positive-definite function. 
 Conversely, every continuous positive-definite function is a Fourier transform of a positive and finite Radon measure.  
 \label{thm:Bochner}
\end{theorem} 

Recall that a function $\psi\from \R^n \to \C$ is called \emph{negative-definite} if $\psi(0) \geq 0$ and for every $t \geq 0$ the function $\xi \mapsto \e^{-t\psi(\xi)}$ is positive-definite. 

\begin{corollary}[{\cite[Theorem 3.6.16]{Jacob-01}}]
 Let $(\mu_t)_{t \geq 0}$ be a vaguely continuous semigroup of sub-probability measures on $\R^n$. 
 Then there exists a unique continuous negative-definite function $\psi\from \R^n \to \C$ such that 
 \begin{equation*}
  \fourier \mu_t = \e^{-t \psi} \quad(t\geq 0). 
 \end{equation*}
 Conversely, for every continuous negative-definite function $\psi\from \R^n \to \C$ there exists a vaguely continuous semigroup of sub-probability measures $(\mu_t)_{t\geq0}$ on $\R^n$ such that the above equation holds.
 \label{coro:FourierOfSemigroup}
\end{corollary}

In view of \cref{coro:FourierOfSemigroup} the negative-definite function $\psi\from \R^n \to \C$ corresponding to a vaguely continuous semigroup of sub-probability measures $(\mu_t)_{t\geq0}$ on $\R^n$ is called its \emph{symbol}. 

As in \cref{sec:subordination} there is a \emph{L{\'e}vy--Khinchin formula}. 

\begin{proposition}[{\cite[Theorem 3.7.8]{Jacob-01}}]
 \label{prop:LevyKhinchinFormula}
    Let $\psi\from \R^n\to\C$. The following are equivalent.
    \begin{enumerate}
        \item
            $\psi$ is negative-definite.
        \item
            There exist unique $c \geq 0$, $d \in \R^n$, $Q \in \R^{n \times n}$ positive semi-definite and a unique measure $\mu$ on $\R^n$ satisfying $\mu (\{0\}) = 0$ and $\int_{\R^n} \norm{x}^2 \wedge 1 \, \mu (\d x) < \infty$ such that
            \begin{equation}
            \psi(\xi) = c + i \ip{d}{\xi} + \ip{\xi}{Q \xi} + \int\limits_{\R^n} \Bigl(1 - \e^{-i\ip{x}{\xi}} - \frac{i \ip{x}{\xi}}{1+\norm{x}^2} \Bigr) \mu(\d x) \quad(\xi\in\R^n).
            \label{eq:LevyKhinchinFormula}
            \end{equation}
    \end{enumerate}
\end{proposition}
The measure $\mu$ in \eqref{eq:LevyKhinchinFormula} is called \emph{L{\'e}vy measure} associated with $\psi$.

We will be interested in continuous negative-definite functions $\psi\from\R^n\to\C$ such that $\Re \psi$ tends to $+\infty$ at $\infty$. 
Note that the growth properties of $\Re \psi$ are entirely determined by the corresponding $Q$ and $\mu$. 
In any case $\Re \psi$ will grow if $Q$ is positive definite, i.e., all its eigenvalues are strictly positive because then there exists $\alpha>0$ such that $\Re \psi(\xi) \geq \alpha \norm{\xi}^2$ for all $\xi\in\R^n$.
Hence, we need to concentrate on the integral term in \cref{eq:LevyKhinchinFormula} in case $Q$ is not positive definite. 
The following lemma addresses the case $Q=0$. 

\begin{lemma}
\label{lemma:LevyKhinchinControl}
 Let $\mu$ be a measure on $\R^n$ satisfying $\mu (\{0\}) = 0$ and $\int_{\R^n} \norm{x}^2 \wedge 1 \, \mu (\d x) < \infty$. 
 For $\xi \in \R^n \setminus \{0\}$ define $B_{\xi} := B(0, \tfrac{1}{\norm{\xi}})$ and functions $\psi\from\R^n\to\C$ and $\phi\from \R^n\setminus\{0\}\to\R$ by
 \begin{equation*}
  \psi(\xi):=\int\limits_{\R^n} \Bigl(1 - \e^{-i\ip{x}{\xi}} - \frac{i \ip{x}{\xi}}{1+\norm{x}^2} \Bigr) \mu(\d x), \quad
  \phi(\xi):= \int\limits_{B_{\xi}} \ip{\xi}{x}^{2} \mu (\d x).
 \end{equation*}
 Then
 \begin{equation*}
  \frac{11}{24} \phi(\xi) \leq \Re \psi(\xi) \leq 2 \left( \phi(\xi) + \mu \bigl(\R^n \setminus B_{\xi} \bigr) \right) \quad(\xi\in \R^n\setminus\{0\}). 
  \label{eq:RePsiControl}
 \end{equation*}
 In particular,
 $\lim\limits_{\norm{\xi} \to \infty} \phi(\xi) = \infty$ implies $\lim\limits_{\norm{\xi} \to \infty} \Re \psi(\xi) = \infty$. 
\end{lemma}

\begin{proof}
 Let $\xi\in\R^n\setminus\{0\}$.
 Note that
 \begin{equation*}
  \Re \psi(\xi) = \int\limits_{\R^n} \Bigl(1 - \cos \bigl (\ip{x}{\xi} \bigr) \Bigr) \mu(\d x). 
 \end{equation*}
 Therefore
 \begin{equation}
  \int\limits_{B_{\xi}} \Bigl(1 - \cos \bigl (\ip{x}{\xi} \bigr) \Bigr) \mu(\d x) \leq \Re \psi(\xi) \leq  \int\limits_{B_{\xi}} \Bigl(1 - \cos \bigl (\ip{x}{\xi} \bigr) \Bigr) \mu(\d x) + 2 \mu \bigl(\R^n \setminus B_{\xi} \bigr). 
  \label{eq:InequalityLeibniz}
 \end{equation}
 For $x \in B_{\xi}$ we have that $\bigl( \tfrac{\ip{x}{\xi}^{2k}}{(2k)!} \bigr)_{k \in \N_0}$ is a monotonically decreasing sequence of non-negative numbers with limit $0$. 
 Hence,
 \begin{equation*}
  \frac{11}{24} \ip{x}{\xi}^2\leq \frac{1}{2} \ip{x}{\xi}^2 - \frac{1}{24} \ip{x}{\xi}^4 \leq 1 - \cos \bigl( \ip{x}{\xi} \bigr) \leq \frac{1}{2} \ip{x}{\xi}^2 \leq 2 \ip{x}{\xi}^2. 
  \label{eq:LeibnizEstimate}
 \end{equation*}
Thus, we obtain the assertion.
\end{proof}

\begin{remark}
 \begin{enumerate}
  \item
    If $\mu(\R^n)<\infty$ then $\Re \psi$ is bounded.
  \item
   If $\psi\from \R^n\to\C$ is a continuous negative-definite function with L{\'e}vy--Khinchin representation \eqref{eq:LevyKhinchinFormula}, then \cref{lemma:LevyKhinchinControl} can be easily generalised such that
   \[
    \lim\limits_{\norm{\xi} \to \infty, \xi \in N(Q)} \phi(\xi) = \infty
   \]
   implies $\Re \psi(\xi) \to \infty$ as $\norm{\xi} \to \infty$.
   Here $N(Q)$ denotes the null space of the matrix $Q$ and $\phi\from \R^n\setminus\{0\}\to \R$ is as in \cref{lemma:LevyKhinchinControl}.  
 \end{enumerate} 
\end{remark}

Let $m \in L_{\infty}(\R^n)$ and $p\in[1,\infty)$. 
Then $m$ is called \emph{Fourier multiplier on $L_p(\R^n)$} if $T_m\from \mathcal{S}(\R^n)\to\mathcal{S}(\R^n)'$ given by
\[
 T_m f := \mathcal{F}^{-1} m \mathcal{F}f
\] 
extends to an operator $T_m\in \calL(L_p(\R^n))$ which then is translation invariant. 
Moreover, every Fourier multiplier on $L_p(\R^n)$ is also a Fourier multiplier on $L_q(\R^n)$ for $q \geq p$ (\cite[p.\ 143]{Grafakos-08}). 
In general, it is rather hard to characterise the set of Fourier multipliers on $L_p(\R^n)$. 
However, for $p=2$, Plancherel's theorem yields that the set of Fourier multipliers is precisely given by $L_{\infty}(\R^n)$, while, for $p=1$, the set of Fourier multipliers is the algebra of Fourier transforms of bounded Radon measures on $\R^n$. 
More precisely, let $\mu$ be a bounded Radon measure on $\R^n$. Then $\F\mu \in C_{\mathrm b}(\R^n)$ and the operator $S_\mu\from \mathcal{S}(\R^n)\to \mathcal{S}'(\R^n)$ given by $S_\mu f:= \F^{-1} (\F \mu \cdot \F f) = (2\pi)^{-n/2} \mu \ast f = (2\pi)^{-n/2} \int_{\R^n} f(\cdot-y)\,\mu(\d y)$ has a continuous extension $S_{p,\mu}\in \calL(L_p(\R^n))$ for all $p \in [1,\infty)$, cf.\ \cite[Thm.\ 2.5.8, Thm.\ 2.5.10]{Grafakos-08}.
Thus, given a vaguely continuous convolution semigroup of sub-probability measures $(\mu_t)_{t\geq 0}$ on $\R^n$ and $p\in[1,\infty)$, we obtain a $C_0$-semigroup $(S_p(t))_{t\geq 0}$ on $L_p(\R^n)$ of contractions given by $S_p(t):=S_{p,\mu_t}$ for $t\geq 0$ (see, e.g.\ \cite[Ex.\ 4.6.29]{Jacob-01}).




\begin{lemma}
 Let $p \in (1, \infty)$, $\lambda>0$ and $n \in \N$. Then $\1_{(-\lambda,\lambda)^n}$ is a Fourier multiplier on $L_p(\R^n)$. 
 \label{lemma:FourierMultiplier}
\end{lemma} 

\begin{proof}
 (i)
   For $n=1$ the argument is well known. 
   Namely, by the Mikhlin multiplier theorem (\cite[Thm.\ 5.2.7 (a)]{Grafakos-08}) the function $h := -i \sgn$ is a Fourier multiplier (called the Hilbert transform), and therefore so is
   \[
    \1_{(-\lambda,\lambda)} = 1 - \frac{i}{2} h(\cdot - \lambda) - \frac{i}{2} h(\cdot + \lambda). 
   \]   
  (ii)
    We shall prove the statement for the case $n=2$ for convenience. 
    The proof for general $n$ is exactly the same.
    It is enough to consider functions $f = \sum_{i=1}^k c_i \1_{A_i}\otimes \1_{B_i}$, where the $A_i$'s as well as the $B_i$'s are mesurable with finite measure and pairwise disjoint since their linear span is dense in $L_p(\R^2)$. 
    In the following, for $i=1,2$ we shall write $\F_i$ for the Fourier transform with respect to the first and second coordinate, respectively. Let $C:=\norm{T_{\1_{(-\lambda,\lambda)}}}_{\calL(L_p(\R))}$. 
    Then
    \begin{align*}
       & \norm{\F^{-1} \mathds{1}_{(-\lambda,\lambda) \times (-\lambda,\lambda)} \F f}_{L_p(\R^2)}^p \\ = & \int\limits_{\R} \underbrace{\biggl( \int\limits_{\R} \Bigl| \F_2^{-1} \1_{(-\lambda,\lambda)} \F_2 \sum\limits_i c_i (\F_1^{-1} \1_{(-\lambda,\lambda)} \F_1 \1_{A_i})(x_1) \1_{B_i} \Bigr|^p (x_2) \, \d x_2 \biggr)}_{\leq C^p \int\limits_{\R} \Bigl| \sum\limits_i c_i (\F_1^{-1} \1_{(-\lambda,\lambda)} \F_1 \1_{A_i})(x_1) \1_{B_i} \Bigr|^p (x_2) \, \d x_2} \d x_1 \\
       \leq & C^p \int\limits_{\R} \int\limits_{\R} \sum\limits_i \abs{c_i}^p \abs{(\F_1^{-1} \1_{(-\lambda,\lambda)} \F_1 \1_{A_i})(x_1)}^p \1_{B_i}(x_2) \, \d x_1 \d x_2 \\
       \leq & C^{2p} \int\limits_{\R} \int\limits_{\R} \sum\limits_i \abs{c_i}^p \1_{A_i}(x_1) \1_{B_i}(x_2) \, \d x_1 \d x_2 \\
       = & C^{2p} \norm{f}^p_{L_p(\R^2)}. \qedhere
    \end{align*}
\end{proof}

\begin{remark}
 Note that $L_p(\R^2) = L_p(\R) \otimes L_p(\R)$, which provides an abstract point of view on \cref{lemma:FourierMultiplier}.  
 For more information the reader may consult \cite[7.3, Theorem, p.\ 80]{Defant-93} and the example directly beforehand. 
\end{remark}

For $p\in(1,\infty)$ and $\lambda>0$ let $P_\lambda:=T_{\1_{(-\lambda,\lambda)^n}} \in\calL(L_p(\R^n))$.

\begin{proposition}
 Let $(\mu_t)_{t\geq0}$ be a vaguely continuous convolution semigroup of sub-pro\-ba\-bi\-li\-ty measures on $\R^n$ with related symbol $\psi\from\R^n\to\C$ and $\phi\from\R^n\setminus\{0\}\to\R$ as in \cref{lemma:LevyKhinchinControl}. 
 Let $\bigl(S_{2}(t) \bigr)_{t \geq 0}$ and $(P_\lambda)_{\lambda>0}$ as above, and define $g\from(0,\infty)\to (0,\infty)$ by
 \[
  g(\lambda) :=  \frac{24}{11} \inf_{\xi \in \R^n \setminus (-\lambda,\lambda)^n} \phi(\xi) \quad(\lambda>0).
 \]
 Then for all $\lambda>0$, $t\geq 0$ and $f\in L_2(\R^n)$ we have
 \[
  \norm{(\id-P_{\lambda}) S_{2}(t) f}_{L_2(\R^n)} \leq \e^{-t g(\lambda)} \norm{f}_{L_2(\R^n)}.
 \]
 i.e., $\bigl(S_{2}(t) \bigr)_{t \geq 0}$ fulfills a dissipation estimate. 
 \label{prop:DissipationL2}
\end{proposition}

\begin{proof}
 By Plancherel's theorem the mapping $L_{\infty}(\R^n) \ni m \to \F^{-1} m \, \ast \in \calL \bigl(L_2(\R^n) \bigr)$ is an isometry. 
 Therefore, for $\lambda>0$, $t\geq 0$, and $f\in L_2(\R^n)$ we have
 \[
  \norm{(\id-P_{\lambda})S_{2}(t)f}_{L_2(\R^n)} = \norm{(1-\1_{(-\lambda,\lambda)^n}) \e^{-t \Re \psi}}_{\infty} \norm{f}_{L_2(\R^n)} \leq \e^{-t g(\lambda)} \norm{f}_{L_2(\R^n)},
 \]
 where for the last estimate we used \cref{lemma:LevyKhinchinControl}. 
\end{proof}

In order to get a dissipation estimate for the full range $p \in (1, \infty)$, we apply interpolation as in \cite[Thm.\ 3.3]{GallaunST-20}. 

\begin{corollary}
Let $(\mu_t)_{t\geq 0}$ be a vaguely continuous convolution semigroup of sub-pro\-ba\-bi\-li\-ty measures on $\R^n$, $p\in (1,\infty)$, $\bigl(S_{p}(t) \bigr)_{t \geq 0}$ and $(P_\lambda)_{\lambda>0}$ as above, and $g\from(0,\infty)\to (0,\infty)$ as in \cref{prop:DissipationL2}. 
Then there exist $C\geq 0$ and $\theta\in (0,1)$ such that for all $\lambda>0$, $t\geq0$ and $f\in L_p(\R^n)$ we have
\[
\norm{(\id-P_{\lambda})S_{p}(t)f}_{L_p(\R^n)} \leq C \e^{-t \theta g(\lambda)} \norm{f}_{L_p(\R^n)}.
\]
\end{corollary}

\begin{proof}
    Let $p_0\in [1,\infty)$ such that $p\in (p_0,2)\cup (2,p_0)$. Note that there exists $C_{p_0} \geq 0$ with the property that for all $\lambda>0$ and $t\geq 0$ we have
    \[
    \norm{(\id-P_{\lambda}) S_{p_0}(t)}_{\calL(L_p(\R^n))} \leq C_{p_0}.
    \]
    Therefore, interpolation between $p_0$ and $2$ yields the existence of some $\theta\in (0,1)$ so that
    \begin{align*}
    \norm{(\id-P_{\lambda}) S_{p}(t)f}_{L_p(\R^n)} & \leq \norm{(\id-P_{\lambda}) S_{p_0}(t)}_{\calL(L_p(\R^n))}^{1-\theta} \norm{(\id-P_{\lambda}) S_{2}(t)}_{\calL(L_p(\R^n))}^{\theta} \norm{f}_{L_p(\R^n)} \\
    & \leq C_{p_0}^{1-\theta} \e^{- t \theta g(\lambda)} \norm{f}_{L_p(\R^n)}.
    \end{align*}
    holds for all $f \in L_p(\R^n)$. 
\end{proof}

We will now focus on the uncertainty principle.
Recall that $\thickset \subseteq \R^n$ is called \emph{thick} if $\thickset$ is measurable and there exist $\rho\in (0,1]$ and $L\in (0,\infty)^n$ such that
\[\left\lvert \thickset \cap \left( \bigtimes_{i=1}^n (0,L_i) + x \right) \right\rvert \geq \rho \prod_{i=1}^n L_i \quad(x\in\R^n),\]
where $\lvert \cdot \rvert$ denotes the Lebesgue measure in $\R^n$.
For a given thick set $\thickset\subseteq\R^n$ we consider $C \in \calL \bigl( L_p(\R^n), L_p(\thickset) \bigr)$ defined by $Cf := f|_{\thickset}$. 

\begin{proposition}[{Logvinenko--Sereda theorem, see \cite{LogvinenkoS-74,Kovrijkine-01}}]
	\label{prop:Logvinenko-Sereda_Rd}
	Let $\thickset \subseteq \R^n$ be a thick set. Then there exist $d_0,d_1>0$ such that for all $p\in [1,\infty]$ and all $f\in L_p(\R^n)$ satisfying $\supp \F f \subseteq [-\lambda,\lambda]^n$ we have 
	\[
	\norm{f}_{L_p(\R^n)} \leq  d_0 \euler^{d_1\lambda}  \norm{f}_{L_p(\thickset)}.
	\]
\end{proposition}

Combined with the dissipation estimate above, we obtain the following immediate consequence from \cref{thm:spectral+diss-obs}.

\begin{theorem}
 \label{thm:ObservabilityConvSG}
 Let $( \mu_t )_{t \geq 0}$ be a vaguely continuous convolution semigroup of sub-probability measures on $\R^n$, $p \in (1,\infty)$, $\thickset\subseteq \R^n$ a thick set, and $\bigl(S_{p}(t) \bigr)_{t \geq 0}$ and $(P_\lambda)_{\lambda>0}$ as above. 
Further let $T > 0$, $g\from(0,\infty)\to (0,\infty)$ be as in \cref{prop:DissipationL2}, and assume that $\lambda \mapsto \lambda / g(\lambda)$ is monotonically decreasing and such that $1/g$ is integrable at $\infty$.
%
 Then there exists $C_{\mathrm{obs}}\geq 0$ with
 \[\forall\, f\in L_p(\R^n): \quad \norm{S_p(T)f}_{L_p(\R^n)} \leq C_{\mathrm{obs}} \int_0^T \norm{S_p(t)f}_{L_p(\thickset)}\, \d t.\]
\end{theorem}

We finish the paper by having a look at examples and non-examples. 

\begin{example}
We have already seen that in the abscence of a suitable matrix $Q$ the growing properties of a negative-definite function $\psi$ is solely determined by how the corresponding L{\'e}vy measure $\mu$ in the L{\'e}vy--Khinchin formula behaves near $0$.
For this reason let us consider in the example symbols of the form
\[
 \psi(\xi) =  \int\limits_{\R^n} \Bigl(1 - \e^{-i\ip{x}{\xi}} - \frac{i \ip{x}{\xi}}{1+\norm{x}^2} \Bigr) \rho(\norm{x}) \, \d x
\]
with measurable $\rho\from [0,\infty) \to \R$, i.e. $\mu$ is absolutely continuous with respect to Lebesgue measure with a radially symmetric density $\rho$ (and $c=0$, $d=0$, and $Q=0$ in the L{\'e}vy--Khinchin formula). 

 \begin{enumerate}
  \item\label{ex:psi:item:a}
   If we choose $\rho:= \1_{(0,1)}$ then a calculation shows that $g(\lambda) = D \cdot \lambda^{-n}$ ($\lambda>0$) for some suitable constant $D > 0$. 
   Hence, \cref{thm:ObservabilityConvSG} is clearly not applicable. 
   
  \item\label{ex:psi:item:b}
   By \ref{ex:psi:item:a} we sense that we should try to get a singularity for $\rho$ at $0$. 
   For this purpose choose now $\rho\from r \mapsto r^{-n-2+\varepsilon}$ with $\varepsilon > 0$. 
   In this case we find $g(\lambda) = D \cdot \lambda^{2 - \varepsilon}$ ($\lambda>0$) for some $D>0$ and $g$ is increasing as long as $\varepsilon < 2$. 
   Note that $g$ grows slower than $\lambda\mapsto \lambda^2$ (however the growth can get arbitrarily close to it) in accordance with general results on L{\'e}vy processes.
   Moreover, \cref{thm:ObservabilityConvSG} is applicable if $\varepsilon < 1$. 
   
  \item
   A special case of (a scaled version of) \ref{ex:psi:item:b} is given by $\varepsilon=2-2\alpha$ for $\alpha \in (\frac{1}{2},1)$. Thus, we recover the statement of \cref{ex:FracLaplace} since
   \[
    \psi(\xi) = \int\limits_{\R^n} \Bigl(1 - \e^{-i\ip{x}{\xi}} - \frac{i \ip{x}{\xi}}{1+\norm{x}^2} \Bigr) \rho(\norm{x}) \, \d x = \norm{\xi}^{2\alpha} 
   \]
   for
   \[\rho(r) = \frac{4^{\alpha}\Gamma(\frac{n}{2}+\alpha)}{\pi^{\frac{n}{2}} \Gamma(-\alpha)} r^{-n-2\alpha},\]
   as a longer calculation shows. 
   We see that using \cref{lemma:LevyKhinchinControl} instead of the explicit knowledge of the symbol $\psi$ yields the same result. 
 \end{enumerate}
\end{example}


\begin{thebibliography}{LWXY20}

\bibitem[BPS18]{BeauchardP-18}
K.~Beauchard and K.~Pravda-Starov.
\newblock Null-controllability of hypoelliptic qua\-dra\-tic differential
  equations.
\newblock {\em J. \'Ec. polytech. Math.}, 5:1--43, 2018.

\bibitem[Boch49]{Bochner-49}
S.~Bochner.
{Diffusion Equation and Stochastic Processes}.
\newblock Proc. Natl. Acad. Sci. USA 35(7):368--370, 1949.

\bibitem[BGST21]{BombachGST-20}
C.~Bombach, D.~Gallaun, C.~Seifert and M.~Tautenhahn.
\newblock Observability and null-controllability for parabolic equations in
  $l_p$-spaces.
\newblock arXiv:2005.14503, 2021.

\bibitem[Car88]{Carja-88}
O.~Carja.
\newblock On constraint controllability of linear systems in {B}anach spaces.
\newblock {\em J. Optim. Theory Appl.}, 56(2):215--225, 1988.

\bibitem[DF93]{Defant-93}
A.~Defant and K.~Floret
\newblock \textit{Tensor Norms and Operator Ideals,} volume 176 of \textit{North-Holland Mathematics Studies.}
\newblock North-Holland, Amsterdam, 1\textsuperscript{st} edition, 1993.

\bibitem[Dou66]{Douglas-66}
{R.G.}~Douglas.
\newblock On majorization, factorization, and range inclusion of operators on
  hilbert space.
\newblock {\em Proc. Amer. Math. Soc.}, 2(17):413--415, 1966.

\bibitem[DR77]{DoleckiR-77}
S.~Dolecki and {D.L.}~Russell.
\newblock A general theory of observation and control.
\newblock {\em SIAM J. Control Optim.}, 2(15):185--220, 1977.

\bibitem[DM12]{DuyckaertsM-12}
T.~Duyckaerts and L.~Miller.
\newblock Resolvent conditions for the control of parabolic equations.
\newblock {\em Journal of Functional Analysis}, 263(11):3641-3673, 2012.

\bibitem[EV18]{EgidiV-18}
M.~Egidi and I.~Veseli{\'c}.
\newblock Sharp geometric condition for null-controllability of the heat
  equation on $\mathbb{R}^d$ and consistent estimates on the control cost.
\newblock {\em Arch. Math. (Basel)}, 111(1):1--15, 2018.

\bibitem[GST20]{GallaunST-20}
D.~Gallaun, C.~Seifert and M.~Tautenhahn.
\newblock Sufficient criteria and sharp geometric conditions for observability
  in {B}anach spaces.
\newblock {\em SIAM J. Control Optim.}, 58(4):2639--2657, 2020.

\bibitem[Gra08]{Grafakos-08}
L.~Grafakos.
\newblock \textit{Classical Fourier analysis,} volume 249 of \textit{Graduate Texts in Mathematics.}
\newblock Springer, New York, 2\textsuperscript{nd} edition, 2008.

\bibitem[Jac01]{Jacob-01}
N.~Jacob. 
\newblock \textit{Pseudo Differential Operators and Markov Processes,} volume I. World Scientific, 2001.

\bibitem[JL99]{JerisonL-99}
D.~Jerison and G.~Lebeau.
\newblock Nodal sets of sums of eigenfunctions.
\newblock In M.~Christ, {C.~E.} Kenig, and C.~Sadosky, editors, {\em Harmonic
  Analysis and Partial Differential Equations}, Chicago Lectures in
  Mathematics, pages 223--239. University of Chicago Press, Chicago, IL, 1999.

\bibitem[Koe20]{Koenig2020}
A.~Koenig.
\newblock Lack of null-controllability for the fractional heat equation and related equations.
\newblock {\em SIAM J. Control Optim.}, 58(6): 3130--3160, 2020.

\bibitem[Kov01]{Kovrijkine-01}
O.~Kovrijkine.
\newblock Some results related to the {L}ogvinenko-{S}ereda {T}heorem.
\newblock {\em Proc. Amer. Math. Soc.}, 129(10):3037--3047, 2001.

\bibitem[KMS21]{KruseMS-21}
K.~Kruse, J.~Meichsner and C.~Seifert.
\newblock Subordination for sequentially equicontinuous equibounded $ C_0
  $-semigroups.
\newblock {\em J. Evol. Equ.}, 21(2):2665--2690, 2021.

\bibitem[LR95]{LebeauR-95}
G.~Lebeau and L.~Robbiano.
\newblock Contr{\^o}le exact de l'{\'e}quation de la chaleur.
\newblock {\em Comm. Partial Differential Equations}, 20(1--2):335--356, 1995.

\bibitem[LZ98]{LebeauZ-98}
G.~Lebeau and E.~Zuazua.
\newblock Null-controllability of a system of linear thermoelasticity.
\newblock {\em Arch. Ration. Mech. Anal.}, 141(4):297--329, 1998.


\bibitem[LS74]{LogvinenkoS-74}
{V.N.}~Logvinenko and {Ju.F.}~Sereda.
\newblock Equivalent norms in spaces of entire functions of exponential type.
\newblock {\em Teor. Funkts., Funkts. anal. Prilozh.}, 20:102--111, 1974.

\bibitem[Mil04]{Miller-04}
L.~Miller.
\newblock Geometric bounds on the growth rate of null-controllability cost for
  the heat equation in small time.
\newblock {\em J. Differential Equations}, 204(1):202--226, 2004.

\bibitem[Mil10]{Miller-10}
L.~Miller.
\newblock A direct {L}ebeau-{R}obbiano strategy for the observability of
  heat-like semigroups.
\newblock {\em Discrete Contin. Dyn. Syst. Ser. B}, 14(4):1465--1485, 2010.

\bibitem[NTTV20]{NakicTTV-20}
I.~Naki\'c, M.~T\"aufer, M.~Tautenhahn and I.~Veseli\'c.
\newblock Sharp estimates and homogenization of the control cost of the heat
  equation on large domains.
\newblock {\em ESAIM Control Optim. Calc. Var.}, 26(54):26 pages, 2020.

\bibitem[Phi52]{Phillips1952}
R.S.~Phillips,
\newblock {On the generation of semigroups of linear operators}.
\newblock {\em Pacific J. Math.}, 2(3): 343--369, 1952.

\bibitem[SSV12]{Schilling-12}
R.L.~Schilling, R.~Song and Z.~Vondracek. 
\newblock \textit{Bernstein Functions: Theory and
Applications,} volume 37 of \textit{de Gruyter Stud. Math.}
\newblock de Gruyter, Berlin, 2\textsuperscript{nd} edition, 2012.

\bibitem[TT11]{TenenbaumT-11}
G.~Tenenbaum and M.~Tucsnak.
\newblock On the null-controllability of diffusion equations.
\newblock {\em ESAIM Control Optim. Calc. Var.}, 17(4):1088--1100, 2011.

\bibitem[Vie05]{Vieru-05}
A.~Vieru.
\newblock On null controllability of linear systems in {B}anach spaces.
\newblock {\em Systems Control Lett.}, 54(4):331--337, 2005.

\bibitem[WWZZ19]{WangWZZ-19}
G.~Wang, M.~Wang, C.~Zhang and Y.~Zhang.
\newblock Observable set, observability, interpolation inequality and spectral
  inequality for the heat equation in $\mathbb{R}^n$.
\newblock {\em J. Math. Pures Appl.}, 126:144--194, 2019.

\end{thebibliography}

\end{document}